\documentclass[12pt]{article}

\usepackage{amsmath,amssymb,amsthm,amscd,amsfonts}
\usepackage{epsfig,pinlabel,paralist}
\usepackage{mathtools}
\usepackage[vmargin=1in, hmargin=1.25in]{geometry}
\usepackage[font=small,format=plain,labelfont=bf,up,textfont=it,up]{caption}
\usepackage{calc}
\usepackage{etoolbox}
\usepackage{microtype}
\usepackage[all,cmtip]{xy}

\usepackage[bookmarks, bookmarksdepth=2, colorlinks=true, linkcolor=blue, citecolor=blue, urlcolor=blue]{hyperref}

\apptocmd{\thebibliography}{\raggedright}{}{}

\numberwithin{equation}{section}

\theoremstyle{plain}
\newtheorem{theorem}{Theorem}[section]
\newtheorem{maintheorem}{Theorem}

\newtheorem{lemma}[theorem]{Lemma}

\newtheorem{question}[theorem]{Question}

\theoremstyle{definition}

\newtheorem{defn}[theorem]{Definition}
\newenvironment{definition}[1][]{\begin{defn}[#1]\pushQED{\qed}}{\popQED \end{defn}}

\theoremstyle{remark}
\newtheorem{rmk}[theorem]{Remark}
\newenvironment{remark}[1][]{\begin{rmk}[#1] \pushQED{\qed}}{\popQED \end{rmk}}

\newtheorem{eg}[theorem]{Example}
\newenvironment{example}[1][]{\begin{eg}[#1] \pushQED{\qed}}{\popQED \end{eg}}

\newcommand\R{\ensuremath{\mathbb{R}}}

\newcommand\Z{\ensuremath{\mathbb{Z}}}

\DeclareMathOperator{\HH}{H}

\newcommand\Set[2]{\ensuremath{\left\{\text{#1 $|$ #2}\right\}}}
\newcommand\GroupPres[2]{\ensuremath{\left\langle \text{#1 $|$ #2} \right\rangle}}

\newcommand\fT{\ensuremath{\mathfrak{T}}}

\newcommand\tX{\ensuremath{\widetilde{X}}}

\newcommand\tp{\ensuremath{\widetilde{p}}}

\newcommand\tSigma{\ensuremath{\widetilde{\Sigma}}}
\newcommand\tast{\ensuremath{\widetilde{\ast}}}

\newcommand\oB{\ensuremath{\overline{B}}}

\newcommand\oS{\ensuremath{\overline{S}}}

\newcommand\op{\ensuremath{\overline{p}}}

\newcommand\ox{\ensuremath{\overline{x}}}

\newcommand\orho{\ensuremath{\overline{\rho}}}

\newcommand\hL{\ensuremath{\widehat{L}}}

\newcommand\Figure[4]{
\begin{figure}[t]
\centering
\centerline{\psfig{file=#2,scale=#4}}
\caption{#3}
\label{#1}
\end{figure}}

\DeclareMathOperator{\ab}{ab}
\newcommand\BigFreeProd{\mathop{\mbox{\Huge{$\ast$}}}}

\title{\vspace{-40pt}The commutator subgroups of free groups and surface groups\vspace{-15pt}}
\author{Andrew Putman\thanks{Supported in part by NSF grant DMS-1811210}}
\date{}

\begin{document}

\vspace{-10pt}
\maketitle

\vspace{-18pt}
\begin{abstract}
\noindent
A beautifully simple free generating set for the commutator subgroup of a free group
was constructed by Tomaszewski.  We give a new geometric proof of his theorem, and show
how to give a similar free generating set for the commutator subgroup of a surface group.
We also give a simple representation-theoretic description of the structure of the 
abelianizations of these commutator subgroups and calculate their homology.
\end{abstract}

\section{Introduction}
\label{section:introduction}

Let $F_n$ be the free group on $\{x_1,\ldots,x_n\}$.  The commutator subgroup
$[F_n,F_n]$ is an infinite-rank free group.  The conjugation action
of $F_n$ on $[F_n,F_n]$ induces an action on its abelianization $[F_n,F_n]^{\ab}$
that factors through $F_n/[F_n,F_n] = \Z^n$, making $[F_n,F_n]^{\ab}$ into
a $\Z[\Z^n]$-module.  This paper addresses the following questions:
\begin{itemize}
\item What kind of $\Z[\Z^n]$-module is $[F_n,F_n]^{\ab}$?  Can we describe it in
terms of simpler representations and calculate its homology?
\item It is easy to write down free generating sets for $[F_n,F_n]$; indeed, this is
often given as an exercise when teaching covering spaces.  However, doing this naively
gives a very complicated free generating set.  Can it be made simpler?
\end{itemize}
We also study the analogous questions for the fundamental groups of closed surfaces.

\Figure{figure:ranktwo}{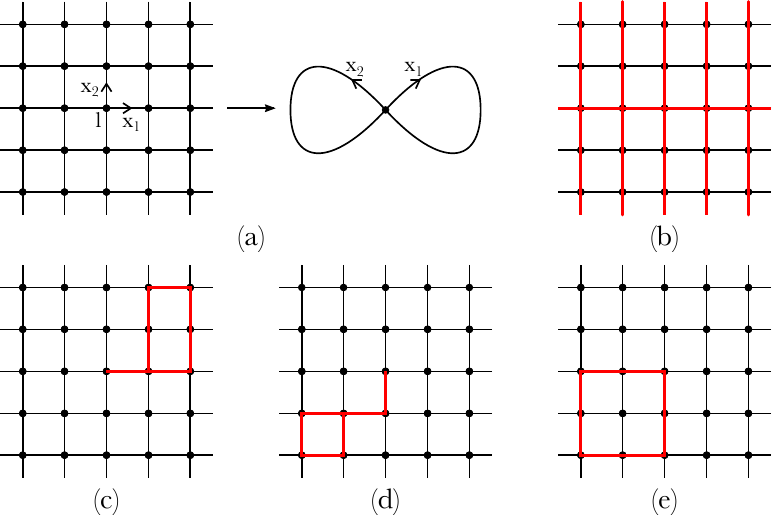}{(a) The cover $\tX \rightarrow X$ corresponding to $[F_2,F_2]$.
(b) A maximal tree in $\tX$.
(c) A generator $x_1 x_2^2 x_1 x_2^{-2} x_1^{-2} \in [F_2,F_2]$ coming from the maximal tree.
(d) The element $[x_1,x_2]^{x_1 x_2} \in [F_2,F_2]$.
(e) The element $[x_1^2, x_2^2] \in [F_2,F_2]$.}{82}

\paragraph{Rank 2.}
As motivation, we start with the case $n=2$.  Regarding
$F_2$ as the fundamental group of a wedge $X$ of two circles, the subgroup $[F_2,F_2]$
corresponds to the cover $\tX$ shown in Figure \ref{figure:ranktwo}.a.
The abelianization $[F_2,F_2]^{\ab}$ is $\HH_1(\tX)$.  This is a free abelian group
with basis the set of squares in Figure \ref{figure:ranktwo}.a.  The deck group $\Z^2$ permutes these squares
simply transitively, so $[F_2,F_2]^{\ab}$ is a rank-$1$ free $\Z[\Z^2]$-module:
\[[F_2,F_2]^{\ab} \cong \Z[\Z^2].\]
There are several natural choices of free generating sets for the free group $[F_2,F_2]$:
\begin{itemize}
\item[(i)] Using the maximal tree shown in Figure \ref{figure:ranktwo}.b, we obtain
the free generating set
\[\Set{$x_1^{k_1} x_2^{k_2} x_1 x_2^{-k_2} x_1^{-k_1-1}$}{$k_1,k_2 \in \Z$ and $k_2 \neq 0$}\]
illustrated in Figure \ref{figure:ranktwo}.c.
\item[(ii)] It is not hard\footnote{We will prove a much more general result
in Theorem \ref{maintheorem:commfnbasis} below.} to convert the free generating set (i) into
the free generating set
\[\Set{$[x_1,x_2]^{x_1^{k_1} x_2^{k_2}}$}{$k_1,k_2 \in \Z$}.\]
illustrated in Figure \ref{figure:ranktwo}.d.  Here our conventions
are that $[y,z] = y^{-1} z^{-1} y z$ and that superscripts indicate
conjugation: $y^z = z^{-1} y z$.  Unlike in (i), this
does {\em not} arise from a maximal tree.
\item[(iii)] A classical exercise in combinatorial group theory (see \cite[p.\ 196, exer.\ 24]{MagnusKarassSolitar}
or \cite[Proposition I.1.4]{SerreTrees}) gives the free generating set
\[\Set{$[x_1^{k_1}, x_2^{k_2}]$}{$k_1,k_2 \in \Z$ and $k_1,k_2 \neq 0$}\]
illustrated in Figure \ref{figure:ranktwo}.e.  Again, this does not arise
from a maximal tree.
\end{itemize}

\begin{remark}
It is hard to see that
$[F_2,F_2]^{\ab} \cong \Z[\Z^2]$
from the generating sets (i) and (iii) above, but this isomorphism can be easily
deduced from the generating set (ii).
\end{remark}

\paragraph{Generating sets in higher rank.}
How about $[F_n,F_n]$ for $n \geq 3$?  A maximal tree argument like (i) above
leads to the very complicated free generating set
\[\Set{$(x_1^{k_1} \cdots x_m^{k_m}) x_{\ell} (x_1^{k_1} \cdots x_{\ell}^{k_{\ell}+1} \cdots x_m^{k_m})^{-1}$}{$1 \leq \ell < m \leq n$, $k_1,\ldots,k_m \in \Z$, $k_m \neq 0$}.\]
Tomaszewski generalized (ii) to give the following much simpler free generating set:

\begin{maintheorem}[Tomaszewski \cite{Tomaszewski}]
\label{maintheorem:commfnbasis}
For $n \geq 1$, the group $[F_n,F_n]$ is freely generated by
\[\Set{$[x_i,x_j]^{x_i^{k_i} x_{i+1}^{k_{i+1}} \cdots x_n^{k_n}}$}{$1 \leq i < j \leq n$ and $k_i,\ldots,k_n \in \Z$}.\]
\end{maintheorem}

Theorem \ref{maintheorem:commfnbasis} plays an important role in the author's work on the Torelli
group \cite{PutmanInfinite, PutmanSmallGenset}.
Tomaszewski's proof of it was combinatorial and involved extensive calculations
with commutator identities.  We will give a proof that is geometric and calculation-free.

\begin{remark}
It is unclear how to generalize the generating set (iii) above to $n \geq 3$.
\end{remark}

\paragraph{Generating sets for surface groups.}
Consider the fundamental group of a closed oriented genus-$g$ surface $\Sigma_g$:
\[\pi_1(\Sigma_g) = \GroupPres{$x_1,\ldots,x_{2g}$}{$[x_1,x_2] \cdots [x_{2g-1}, x_{2g}] = 1$}.\]
Just like $[F_n,F_n]$, the commutator subgroup $[\pi_1(\Sigma_g),\pi_1(\Sigma_g)]$ is
an infinite-rank free group\footnote{Letting $\tSigma$ be the cover of $\Sigma_g$ corresponding
to $[\pi_1(\Sigma_g),\pi_1(\Sigma_g)]$, this means that $\pi_1(\tSigma)$ is free.
In fact, Johansson \cite{Johansson} proved that all non-compact connected surfaces have free fundamental
groups.  A conceptual way to see this uses a theorem of Whitehead (\cite{Whitehead}; see \cite{PutmanNote} for an 
expository account) that says that any smooth connected noncompact $n$-manifold $M$ deformation retracts
to an $(n-1)$-dimensional subcomplex.  Applied to $\tSigma$, this says that $\tSigma$ deformation retracts
to a $1$-dimensional simplicial complex, i.e.\ a graph.  We conclude by observing that graphs have free fundamental
groups.}.  How can we find a free generating set for it?  

One idea is to try to use Theorem \ref{maintheorem:commfnbasis}.  The group $[\pi_1(\Sigma_g),\pi_1(\Sigma_g)]$
can be obtained from $[F_{2g},F_{2g}]$ by quotienting out by the $F_{2g}$-normal closure of the surface
relation
\[[x_1,x_2] \cdots [x_{2g-1}, x_{2g}].\]
When you try to express the $F_{2g}$-conjugates of this in terms of the free basis
\[\Set{$[x_i,x_j]^{x_i^{k_i} x_{i+1}^{k_{i+1}} \cdots x_{2g}^{k_{2g}}}$}{$1 \leq i < j \leq 2g$ and $k_i,\ldots,k_{2g} \in \Z$}\]
for $[F_{2g},F_{2g}]$ given by Theorem \ref{maintheorem:commfnbasis}, it is natural to assume that
the relations you impose must each involve a generator of the form
\[[x_1,x_2]^{x_1^{k_1} \cdots x_{2g}^{k_{2g}}}.\]
After all, these are the only generators whose conjugating elements involve all of the generators of $F_{2g}$.
This suggests that these $[F_{2g},F_{2g}]$-generators can be eliminated, and the remaining ones should freely
generate $[\pi_1(\Sigma_g),\pi_1(\Sigma_g)]$.

Unfortunately, it seems quite hard to make this idea rigorous.  Even expressing surface relations
of the form
\[\left([x_1,x_2] \cdots [x_{2g-1}, x_{2g}]\right)^{x_1^{k_1} \cdots x_{2g}^{k_{2g}}}\]
in terms of our generators seems quite messy.  Nonetheless, we will give a different geometric
argument to show that the above idea does lead to a free generating set:

\begin{maintheorem}
\label{maintheorem:commsurfbasis}
For $g \geq 1$, the group $[\pi_1(\Sigma_g),\pi_1(\Sigma_g)]$ is freely generated by
\[\Set{$[x_i,x_j]^{x_i^{k_i} x_{i+1}^{k_{i+1}} \cdots x_{2g}^{k_{2g}}}$}{$1 \leq i < j \leq 2g$, $(i,j) \neq (1,2)$, and $k_i,k_{i+1},\ldots,k_{2g} \in \Z$}.\]
\end{maintheorem}

\begin{remark}
We do not know any other explicit free generating sets in the literature for $[\pi_1(\Sigma_g),\pi_1(\Sigma_g)]$.
\end{remark}

\paragraph{Non-freeness and homology.}
As we discussed above, $[F_2,F_2]^{\ab}$ is a free $\Z[\Z^2]$-module.  It is natural to wonder
whether this holds in higher rank.  The structure of the free generating set for $[F_n,F_n]$ given in
Theorem \ref{maintheorem:commfnbasis} suggests that it should not be free, and we prove
that this indeed is the case.

In fact, we prove even more.  For a group $G$, if $M$ is a free $\Z[G]$-module then
$\HH_k(G;M) = 0$ for all $k \geq 1$.  To prove that $M$ is not free, it is thus enough to find
some nonvanishing homology group of degree at least $1$.  We prove the following:

\begin{maintheorem}
\label{maintheorem:homologycalcfn}
For $n \geq 2$ and $k \geq 0$, we have $\HH_k(\Z^n;[F_n,F_n]^{\ab}) \cong \wedge^{k+2} \Z^n$.  In particular,
$[F_n,F_n]^{\ab}$ is not a free $\Z[\Z^n]$-module for $n \geq 3$.
\end{maintheorem}

\begin{remark}
The case $k=0$ says that $\HH_0(\Z^n;[F_n,F_n]^{\ab}) \cong \wedge^2 \Z^n$.  Recall that the $0^{\text{th}}$
homology group is the coinvariants of the coefficient module, i.e.\ the largest
quotient of the coefficient module on which the group acts trivially.  Since the action of
$\Z^n$ on $[F_n,F_n]^{\ab}$ is induced by the conjugation action of $F_n$ on $[F_n,F_n]$, we have
\[\HH_0(\Z^n;[F_n,F_n]^{\ab}) = \frac{[F_n,F_n]}{[F_n,[F_n,F_n]]} \cong \wedge^2 \Z^n.\]
This is a special case of a classical theorem of Magnus (\cite{MagnusLie}; see \cite{SerreLie} for
an expository account) that says that the graded quotients of the lower central series of a free group
are precisely the graded terms of the free Lie algebra.
\end{remark}

For surface groups, the abelianization $[\pi_1(\Sigma_g),\pi_1(\Sigma_g)]^{\ab}$ is a module over
\[\frac{\pi_1(\Sigma_g)}{[\pi_1(\Sigma_g),\pi_1(\Sigma_g)]} \cong \HH_1(\Sigma_g) \cong \Z^{2g}.\]
The analogue of Theorem \ref{maintheorem:homologycalcfn} for these groups is as follows:

\begin{maintheorem}
\label{maintheorem:homologycalcsurf}
For $g \geq 1$ and $k \geq 0$, we have
\[\HH_k(\Z^{2g};[\pi_1(\Sigma_g),\pi_1(\Sigma_g)]^{\ab}) \cong
\begin{cases}
(\wedge^2 \Z^{2g}) / \Z & \text{if $k = 0$},\\
\wedge^{k+2} \Z^{2g} & \text{if $k \geq 1$}.
\end{cases}\]
In particular, $[\pi_1(\Sigma_g),\pi_1(\Sigma_g)]^{\ab}$ is not a free $\Z[\Z^{2g}]$-module
for $g \geq 2$.
\end{maintheorem}

\paragraph{Module structure.}
If $[F_n,F_n]^{\ab}$ is not a free $\Z[\Z^n]$-module, what kind of module is it?  We will
prove that it has a filtration whose associated graded terms are not free, but are
not too far from being free.  In the following theorem, for $1 \leq m \leq n$ we regard
$\Z^m$ as a module over $\Z^n$ via the projection $\Z^n \rightarrow \Z^m$ onto the final
$m$ coordinates.

\begin{maintheorem}
\label{maintheorem:commfnmod}
For all $n \geq 1$, the $\Z[\Z^n]$-module $[F_n,F_n]^{\ab}$ has a sequence
\[0 = M_0 \subset M_1 \subset \cdots \subset M_{n-1} = [F_n,F_n]^{\ab}\]
of submodules such that
\[M_k/M_{k-1} \cong \left(\Z\left[\Z^{n-k+1}\right]\right)^{\oplus (n-k)}\]
for $1 \leq k \leq n-1$.
\end{maintheorem}

\begin{question}
What is the nature of the extensions in the filtration in Theorem \ref{maintheorem:commfnmod}?  Do any
of them split as $\Z[\Z^n]$-modules?
\end{question}

For surface groups, we relate the structures of $[\pi_1(\Sigma_g),\pi_1(\Sigma_g)]^{\ab}$
and $[F_{2g},F_{2g}]^{\ab}$ as follows:

\begin{maintheorem}
\label{maintheorem:commsurfmod}
For all $g \geq 1$, we have a short exact sequence
\[0 \longrightarrow \Z[\Z^{2g}] \longrightarrow [F_{2g},F_{2g}]^{\ab} \longrightarrow [\pi_1(\Sigma_g),\pi_1(\Sigma_g)]^{\ab} \longrightarrow 0\]
of $\Z[\Z^{2g}]$-modules.
\end{maintheorem}

\begin{remark}
As one would expect, the $\Z[\Z^{2g}]$-submodule identified by Theorem \ref{maintheorem:commsurfmod} is generated
by the image of the surface relation and its conjugates.
\end{remark}

\begin{question}
Does the exact sequence of $\Z[\Z^{2g}]$-modules in Theorem \ref{maintheorem:commsurfmod} split?
\end{question}

\paragraph{Outline.}
The rest of this paper is divided into two sections: \S \ref{section:free} contains the proofs
of Theorems \ref{maintheorem:commfnbasis}, \ref{maintheorem:homologycalcfn}, and \ref{maintheorem:commfnmod},
and \S \ref{section:surf} contains the proofs of Theorems \ref{maintheorem:commsurfbasis},
\ref{maintheorem:homologycalcsurf}, and \ref{maintheorem:commsurfmod}.

\paragraph{Conventions.}
For a set $S$, we will let $F(S)$ denote the free group on $S$.  For $S = \{x_1,\ldots,x_n\}$, we
thus have $F(S) = F_n$.  For a group $G$ and $x,y \in G$, we define $[x,y] = x^{-1} y^{-1} x y$ and $x^y = y^{-1} x y$.  This
latter convention ensures that for $x,y,z \in G$ we have $(x^y)^z = x^{yz}$.

\section{Commutator subgroup of free group}
\label{section:free}

In this section, we prove Theorems \ref{maintheorem:commfnbasis}, \ref{maintheorem:homologycalcfn}, 
and \ref{maintheorem:commfnmod}.  

\subsection{Free generating set (Theorem \ref{maintheorem:commfnbasis})}
\label{section:freefn}

We start with some preliminary lemmas.  We emphasize in 
these lemmas that we allow free groups to have infinite rank.  Our first lemma
will only be used in an example in this section, but will be needed in a more serious
way when we discuss the commutator subgroup of a surface group.

\begin{lemma}
\label{lemma:easyargument}
Let $T$ be a set and let $S \subset T$.  The normal closure of $S$ in $F(T)$ is then freely
generated by
\[\Set{$s^w$}{$s \in S$ and $w \in F(T \setminus S)$}.\]
\end{lemma}
\begin{proof}
This is a standard consequence of covering space theory applied to a wedge of $|T|$ circles, whose
fundamental group can be identified with $F(T)$.
\end{proof}

\begin{lemma}
\label{lemma:recognizefree}
Let $F$ be a free group and let $S \subset F$ be a generating set that projects to a basis for the
free abelian group $F^{\ab}$.  Then $F$ is freely generated by $S$.
\end{lemma}
\begin{proof}
To show that no nontrivial reduced word in $S$ is trivial in $F$, it suffices to prove that
all finite subsets $S' \subset S$ freely generate the subgroup $F'$ they generate.
The composition
\[F' \longrightarrow (F')^{\ab} \longrightarrow F^{\ab}\]
takes $S'$ to a linearly independent subset of $F^{\ab}$.  Since the image of $S'$ in $(F')^{\ab}$ generates it, we
deduce that $(F')^{\ab}$ is a free abelian group of rank $|S'|$, so $F'$ is a free group of rank $|S'|$.
Thus the map $F(S') \rightarrow F'$ induced by the inclusion $S' \hookrightarrow F'$
is a surjective map between free groups of the same finite rank.  Since finite-rank free groups are
Hopfian\footnote{A group $G$ is Hopfian if all surjective homomorphisms $G \rightarrow G$ are isomorphisms.
More generally, Malcev (\cite{Malcev}; see \cite[Theorem IV.4.10]{LyndonSchupp} for the short proof)
proved that finitely generated residually finite groups are Hopfian.  See \cite{MathOverflowFree} for a
variety of proofs that free groups are residually finite.  My favorite is in \cite{MalesteinPutman}, which
gives a stronger result and builds on a proof of Hempel \cite{HempelResiduallyFinite}.}
the map $F(S') \rightarrow F'$ must be an isomorphism, so $F'$ is freely generated by $S'$.
\end{proof}

\begin{lemma}
\label{lemma:improvesemidirect}
Let $F$ be a free group, let $A < F$ be a subgroup, and let $B \lhd F$ be a normal subgroup such that
$F = B \rtimes A$.  Assume that
$B$ is is freely generated by a set $S$ on which $A$ acts freely by conjugation, 
and let $S' < S$ contain a single element from each
$A$-orbit.  We then have $F = F(S') \ast A$.
\end{lemma}

To help the reader understand this lemma, we give an example before the proof.

\begin{example}
Fix some $1 \leq k < n$ and let $A < F_n$ be the subgroup generated by $\{x_{k+1},\ldots,x_n\}$.
Define $r\colon F_n \rightarrow A$ via the formula
\[f(x_i) = \begin{cases}
1 & \text{if $1 \leq i \leq k$},\\
x_i & \text{if $k+1 \leq i \leq n$}.\end{cases}\]
The homomorphism $r$ is a retraction.  Letting $B = \ker(r)$, we thus have
$F_n = B \rtimes A$.  The group $B$ is the normal closure of
$S' = \{x_1,\ldots,x_k\}$, so Lemma \ref{lemma:easyargument}
says that $B$ is freely generated by
\[S = \Set{$x_i^w$}{$1 \leq i \leq k$ and $w \in A$}.\]
The conjugation action of $A$ on $B$ restricts to a free action on $S$, and
$S' \subset S$ 
contains a single element from each $A$-orbit.  As is asserted in the lemma, we have
$F_n = F(S') \ast A$.
\end{example}

\begin{proof}[Proof of Lemma \ref{lemma:improvesemidirect}]
Let $S'' \subset A$ be a free generating set for $A$.  The set $S' \cup S''$ generates
$F$, and we must prove that it freely generates it.  Since $F = B \rtimes A$, we have
\[F^{\ab} = (B^{\ab})_A \oplus A^{\ab},\]
where the subscript indicates that we are taking the $A$-coinvariants, i.e.\ the largest quotient
on which $A$ acts trivially.  Since $A$ acts freely on $S$ and $S'$ contains a single element from each $A$-orbit,
it follows that $S' \subset B$ projects to a free basis for the free abelian group $(B^{\ab})_A$.  
Since $S''$ projects to a free basis for the free abelian group $A^{\ab}$, we conclude that $S' \cup S''$
projects to a free basis for the free abelian group $F^{\ab}$.  Lemma \ref{lemma:recognizefree} then
implies that $S' \cup S''$ is a free generating set for $F$. 
\end{proof}

Our final lemma contains the geometric heart of Theorem \ref{maintheorem:commfnbasis}.

\begin{lemma}
\label{lemma:recognizeprojection}
Fix some $n \geq 2$.  Let $K \lhd F_n$ be the normal closure of $x_1$ and let $B = K \cap [F_n,F_n]$.
Set $T = \{x_2,\ldots,x_n\}$.  Then $B$ is freely generated by
\[\Set{$[x_1,x_j]^{x_1^{k_1} w}$}{$1 < j \leq n$ and $k_1 \in \Z$ and $w \in F(T)$}.\]
\end{lemma}
\begin{proof}
Let $X$ be the wedge of $n$ circles and let $p \in X$ be the wedge point.  Identify $\pi_1(X,p)$ with
$F_n$.  Let $(\tX,\tp)$ be the cover of $(X,p)$ with $\pi_1(\tX,\tp) = B$.
The group $B$ is the kernel of the surjective homomorphism $f\colon F_n \rightarrow F(T) \times \Z$ defined
via the formula
\[f(x_i) = \begin{cases}
(1,1) & \text{if $i=1$},\\
(x_i,0) & \text{if $2 \leq i \leq n$}.\end{cases}\]
Here the $1$ in the first coordinate of $(1,1) \in F(T) \times \Z$ means the identity element in $F(T)$.
For $1 \leq i \leq n$, let $\ox_i = f(x_i)$.   We can 
identify $\tX$ with the Cayley graph of $F(T) \times \Z$ with respect to the generating set $\{\ox_1,\ldots,\ox_n\}$.
Using this identification, the set of vertices of $\tX$ is identified with $F(T) \times \Z$ and the edges
of $\tX$ are oriented and labeled with elements of $\{\ox_1,\ldots,\ox_n\}$.  The basepoint $\tp$ of $\tX$ is $(1,0) \in F(T) \times \Z$.

Let $Y$ be the Cayley graph of $F(T)$ with respect to $T$.  Identify $Y$ with the full subgraph
of $\tX$ whose vertices are $F(T) \times 0 \subset F(T) \times \Z$.  The projection $F(T) \times \Z \rightarrow F(T)$
induces a surjective map $\rho\colon \tX \rightarrow Y$.  For each edge $e$ of $Y$, let $L_e = \rho^{-1}(e)$.
If $e$ goes from $u \in F(T)$ to $v \in F(T)$, then $L_e$ is the ``ladder'' depicted in Figure \ref{figure:ladder}.a.
The graph $\tX$ is the union of the $L_e$ as $e$ ranges over the edges of $Y$.

\Figure{figure:ladder}{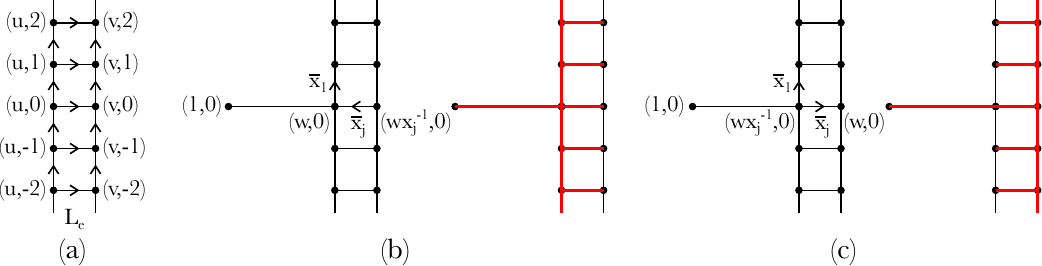}{(a) The ``ladder'' $L_e$ corresponding to the edge $e$.
(b) The space $\hL_e$ with the edge $e$ pointing towards the basepoint, along with the
corresponding maximal tree.
(c) The space $\hL_e$ with the edge $e$ pointing away from the basepoint, along with the
corresponding maximal tree.
}{82}

We want to apply the Seifert--van Kampen theorem to express $\pi_1(\tX,\tp) = B$ in terms of the fundamental
groups of the $L_e$.  For this, we will have to enlarge the $L_e$ so that they include the basepoint $\tp = (1,0)$.
For an edge $e$ of $Y$, let $\hL_e$ be the union of $L_e$ and the shortest edge path in
$Y \subset \tX$ that starts at the basepoint $\tp = (1,0)$ and ends with $e$.  The graph $\hL_e$ contains
$\tp$ and deformation retracts to $L_e$.  

All finite intersections of more than one
of the $\hL_e$ are contractible, so Seifert--van Kampen applies to show that
\[B = \pi_1(\tX,\tp) = \BigFreeProd_{e} \pi_1(\hL_e,\tp).\]
Letting
\[S = \Set{$[x_1,x_j]^{x_1^{k_1} w}$}{$1 < j \leq n$ and $k_1 \in \Z$ and $w \in F(T)$}\]
be our purported free generating set for $B$, it is thus enough to find a subset $S_e \subset S$ for each
edge $e$ such that $\pi_1(\hL_e,\tp)$ is freely generated by $S_e$ and such that
\[S = \bigcup_e S_e.\]
Fix an edge $e$ of $Y$.  Let $\ox_j$ be the label on $e$ and let $(w,0) \in F(T) \times \Z$ be the endpoint of $e$.
We claim that the subgroup $\pi_1(\hL_e,\tp)$ of $\pi_1(\tX,\tp) = B$ is freely generated by
\[S_e = \Set{$[x_1,x_j]^{x_1^{k_1} w^{-1}}$}{$k_1 \in \Z$}.\]
Here we emphasize that our convention is that
\[[x_1,x_j]^{x_1^{k_1} w^{-1}} = w x_1^{-k_1} x_1^{-1} x_j^{-1} x_1 x_j x_1^{k_1} w.\]
The proof of this claim divides into two cases:
\begin{itemize}
\item The first is that $e$ points towards the basepoint as in Figure \ref{figure:ladder}.b.  The purported
free generating set is then the one associated to the maximal tree
depicted in Figure \ref{figure:ladder}.b.
\item The second is that $e$ points away from the basepoint as in Figure \ref{figure:ladder}.c.  The purported
free generating set is then the one associated to the maximal tree
depicted in Figure \ref{figure:ladder}.c.
\end{itemize}
Since $S$ is the union of the $S_e$, the proof is complete.
\end{proof}

We now prove Theorem \ref{maintheorem:commfnbasis}.

\begin{proof}[Proof of Theorem \ref{maintheorem:commfnbasis}]
Recall that we must prove that the group $[F_n,F_n]$ is freely generated by
\[\Set{$[x_i,x_j]^{x_i^{k_i} x_{i+1}^{k_{i+1}} \cdots x_n^{k_n}}$}{$1 \leq i < j \leq n$ and $k_i,\ldots,k_n \in \Z$}.\]
The proof will be by induction on $n$.  The base case $n=1$ is trivial since $[F_1,F_1] = 1$ and the purported
free generating set is empty.  

Assume now that $n>1$ and that the result is true for $[F_{n-1},F_{n-1}]$.  Set
\[T = \{x_2,\ldots,x_n\} \subset F_n,\]
and let $r\colon F_n \rightarrow F(T)$ be the retraction defined via the formula
\[r(x_i) = \begin{cases}
1 & \text{if $i=1$},\\
x_i & \text{if $2 \leq i \leq n$}.\end{cases}\]
Set $K = \ker(r)$, so $K$ is the normal closure of $x_1$.  We have a split
short exact sequence
\begin{equation}
\label{eqn:splitft}
1 \longrightarrow K \longrightarrow F_n \stackrel{r}{\longrightarrow} F(T) \longrightarrow 1,
\end{equation}
so $F_n = K \rtimes F(T)$.  

Since $r([F_n,F_n]) = [F(T),F(T)]$, the exact sequence \eqref{eqn:splitft} restricts to a
short exact sequence
\[1 \longrightarrow K \cap [F_n,F_n] \longrightarrow [F_n,F_n] \stackrel{r}{\longrightarrow} [F(T),F(T)] \longrightarrow 1.\]
This short exact sequence also splits.  Letting
\[A = [F(T),F(T)] \quad \text{and} \quad B = K \cap [F_n,F_n],\]
we therefore have $[F_n,F_n] = B \rtimes A$.

By Lemma \ref{lemma:recognizeprojection}, the group $B$ is freely generated by
\[S = \Set{$[x_1,x_j]^{x_1^{k_1} w}$}{$1 < j \leq n$ and $k_1 \in \Z$ and $w \in F(T)$}.\]
The conjugation action of $F(T)$ on $B$ preserves $S$ and acts freely on it.  Restrict this
action to $A = [F(T),F(T)]$.  The group $A$ still acts freely on $S$, and the set
\[S' = \Set{$[x_1,x_j]^{x_1^{k_1} x_2^{k_2} \cdots x_n^{k_n}}$}{$1 < j \leq n$ and $k_1,\ldots,k_n \in \Z$}\]
contains exactly one representative from each orbit of the action of $A$ on $S$.  Lemma
\ref{lemma:improvesemidirect} then implies that
\[[F_n,F_n] = F(S') \ast A.\]
By induction, $A = [F(T),F(T)]$ is freely generated by
\[S'' = \Set{$[x_i,x_j]^{x_i^{k_i} x_{i+1}^{k_{i+1}} \cdots x_n^{k_n}}$}{$2 \leq i < j \leq n$ and $k_i,\ldots,k_n \in \Z$}.\]
We conclude that $[F_n,F_n]$ is freely generated by $S' \cup S''$, as desired. 
\end{proof}

\subsection{Module structure (Theorem \ref{maintheorem:commfnmod})}
\label{section:modfn}

We start with the following observation.

\begin{lemma}
\label{lemma:commutator}
Let $G$ be a group and let $x,y,w,w' \in G$.  Assume that $w$ and $w'$ map to the same
element of $G^{\ab}$.  Then $[x,y]^w$ and $[x,y]^{w'}$ map to the same element of $[G,G]^{\ab}$.
\end{lemma}
\begin{proof}
For $z,z' \in [G,G]$, write $z \equiv z'$ if $z$ and $z'$ map to the same element of $[G,G]^{\ab}$.
We can write $w' = w c$ with $c \in [G,G]$, so
\[[x,y]^{w'} = c^{-1} [x,y]^w c \equiv c^{-1} c [x,y]^w = [x,y]^w.\qedhere\]
\end{proof}

This allows us to make the following definition.

\begin{definition}
Let $G$ be a group.  For $x,y \in G$ and $h \in G^{\ab}$, write $\{x,y\}^h$ for the image
in $[G,G]^{\ab}$ of $[x,y]^w$, where $w \in G$ is any element projecting to $h$.
\end{definition}

We now launch into the proof of Theorem \ref{maintheorem:commfnmod}.

\begin{proof}[Proof of Theorem \ref{maintheorem:commfnmod}]
Recall that the theorem in question asserts that for all $n \geq 1$, the
$\Z[\Z^n]$-module $[F_n,F_n]^{\ab}$ has a sequence
\[0 = M_0 \subset M_1 \subset \cdots \subset M_{n-1} = [F_n,F_n]^{\ab}\]
of submodules such that
\[M_k/M_{k-1} \cong \left(\Z\left[\Z^{n-k+1}\right]\right)^{\oplus (n-k)}\]
for $1 \leq k \leq n-1$.  The proof of this will be by induction on $n$.
The base case $n=1$ is trivial since $[F_1,F_1] = 1$ and the theorem
simply asserts that $[F_1,F_1]^{\ab} = 0$.

Assume now that $n>1$ and that the theorem is true for $F_{n-1}$.  Let $T = \{x_2,\ldots,x_n\}$ and
let $r\colon F_n \rightarrow F(T)$ be the evident retract whose kernel is the normal closure of $x_1$.  The
map $r$ restricts to a surjective homomorphism $p\colon [F_n,F_n] \rightarrow [F(T),F(T)]$.  Let
$B = \ker(p)$.  Let $\op \colon [F_n,F_n]^{\ab} \rightarrow [F(T),F(T)]^{\ab}$ be the map induced
by $p$ and let $\oB = \ker(\orho)$.  

The map $r$ descends to the projection $\Z^n \rightarrow \Z^{n-1}$ onto the last
$(n-1)$ coordinates.  Use this projection to make $[F(T),F(T)]^{\ab}$ into a module
over $\Z[\Z^n]$.  We then have a short exact sequence
\begin{equation}
\label{eqn:abshort}
0 \longrightarrow \oB \longrightarrow [F_n,F_n]^{\ab} \stackrel{\op}{\longrightarrow} [F(T),F(T)]^{\ab} \longrightarrow 0
\end{equation}
of $\Z[\Z^n]$-modules.  
Since $|T|=n-1$, our inductive hypothesis
says that $[F(T),F(T)]^{\ab}$ has a filtration of the appropriate form.  Using
\eqref{eqn:abshort}, to prove
the theorem for $[F_n,F_n]^{\ab}$, it is enough to prove that
\[\oB \cong \left(\Z[\Z^n]\right)^{\oplus (n-1)}\]
as $\Z[\Z^n]$-modules.  Indeed, we can then take $M_1 = \oB$ and $M_k$ for $2 \leq k \leq n-1$ to
be the preimage in $[F_n,F_n]^{\ab}$ of the filtration for $[F(T),F(T)]^{\ab}$.

Lemma \ref{lemma:recognizeprojection} says that the group $B$ is freely generated
by
\[S = \Set{$[x_1,x_j]^{x_1^{k_1} w}$}{$1 < j \leq n$ and $k_1 \in \Z$ and $w \in F(T)$}.\]
This implies that $\oB$ is generated as an abelian group by
\[\oS = \Set{$\{x_1,x_j\}^{h}$}{$1 \leq j \leq n$ and $h \in \Z^n$}.\]
Here we are using the notation introduced before the proof.  Theorem \ref{maintheorem:commfnbasis}
says that the set
\[S' = \Set{$[x_1,x_j]^{x_1^{k_1} \cdots x_n^{k_n}}$}{$1 \leq j \leq n$ and $k_1,\ldots,k_n \in \Z$}\]
forms part of a free basis for $[F_n,F_n]$.  This implies that $\oS$ is a basis
for the free abelian group $\oB$.  Since the $\Z[\Z^n]$-action on $[F_n,F_n]^{\ab}$ is induced by the conjugation
action of $F_n$ on $[F_n,F_n]$, this implies that $\oB$ is freely generated
as a $\Z[\Z^n]$-module by the $(n-1)$-element set
\[\Set{$\{x_1,x_j\}$}{$1 \leq j \leq n$}.\]
We conclude that
\[\oB \cong \left(\Z[\Z^n]\right)^{\oplus (n-1)},\]
as desired.
\end{proof}

\subsection{Homology (Theorem \ref{maintheorem:homologycalcfn})}
\label{section:homologyfn}

We now prove Theorem \ref{maintheorem:homologycalcfn}, which calculates the homology of $\Z^n$ with coefficients
in $[F_n,F_n]^{\ab}$.  We could use Theorem \ref{maintheorem:commfnmod}, but will instead give
a shorter direct proof.

\begin{proof}[Proof of Theorem \ref{maintheorem:homologycalcfn}]
Fix some $n \geq 2$.  Our goal is to prove that
\[\HH_k(\Z^n;[F_n,F_n]^{\ab}) \cong \wedge^{k+2} \Z^n \quad \text{for $k \geq 0$}.\]
Consider the short exact sequence
\[1 \longrightarrow [F_n,F_n] \longrightarrow F_n \longrightarrow \Z^n \longrightarrow 1.\]
The associated Hochschild--Serre spectral sequence is of the form
\[E^2_{pq} = \HH_p(\Z^n;\HH_q([F_n,F_n])) \Rightarrow \HH_{p+q}(F_n).\]
Since $[F_n,F_n]$ is a free group, all of these entries vanish except
\[E^2_{p0} = \HH_p(\Z^n) = \wedge^p \Z^n \quad \text{and} \quad E^2_{p1} = \HH_1(\Z^n;[F_n,F_n]^{\ab}).\]
The $E^2$-page of our spectral sequence is thus of the form
\begin{center}
\begin{tabular}{|c@{\hspace{0.2 in}}c@{\hspace{0.2 in}}c@{\hspace{0.2 in}}c@{\hspace{0.2 in}}c}
\footnotesize{$\HH_0(\Z^n;[F_n,F_n]^{\ab})$} & \footnotesize{$\HH_1(\Z^n;[F_n,F_n]^{\ab})$} & \footnotesize{$\HH_2(\Z^n;[F_n,F_n]^{\ab})$} & \footnotesize{$\HH_3(\Z^n;[F_n,F_n]^{\ab})$} & \footnotesize{$\cdots$} \\
\footnotesize{$\Z$} & \footnotesize{$\Z^n$} & \footnotesize{$\wedge^2 \Z^n$} & \footnotesize{$\wedge^3 \Z^n$} & \footnotesize{$\cdots$} \\
\cline{1-5}
\end{tabular}
\end{center}
This has to converge to
\[\HH_k(F_n) = \begin{cases}
\Z & \text{if $k=0$},\\
\Z^n & \text{if $k=1$},\\
0 & \text{otherwise}.\end{cases}\]
We deduce that the differentials
\[\wedge^{k+2} \Z^n = E^2_{k+2,0} \rightarrow E^2_{k,1} = \HH_k(\Z^n;[F_n,F_n]^{\ab})\]
must be isomorphisms for all $k \geq 0$.  The theorem follows.
\end{proof}

\section{Commutator subgroup of surface group}
\label{section:surf}

In this section, we prove Theorems \ref{maintheorem:commsurfbasis}, \ref{maintheorem:homologycalcsurf}, 
and \ref{maintheorem:commsurfmod}.  

\subsection{Free generating set (Theorem \ref{maintheorem:commsurfbasis})}
\label{section:surfbasis}

We begin by introducing some notation that will be used throughout this section:
\begin{itemize}
\item Fix some $g \geq 2$, and let
\[\pi = \pi_1(\Sigma_g) = \GroupPres{$x_1,\ldots,x_{2g}$}{$[x_1,x_2] \cdots [x_{2g-1},x_{2g}]=1$}.\]
\item Let $F < \pi$ be the subgroup generated by $T = \{x_3,x_4,\ldots,x_{2g}\}$.
\item Let $G \lhd \pi$ be the normal closure of $F$.
\end{itemize}
We now prove several preliminary lemmas.

\begin{lemma}
\label{lemma:localcommutator}
We have $[\pi,\pi] < G$.
\end{lemma}
\begin{proof}
The group $G$ is the kernel of the homomorphism
\[f\colon \pi \rightarrow \GroupPres{$x_1,x_2$}{$[x_1,x_2] = 1$} \cong \Z^2\]
defined via the formula
\[f(x_i) = \begin{cases}
x_i & \text{if $i=1,2$},\\
1 & \text{if $3 \leq i \leq 2g$}.\end{cases}\]
Since the codomain of $f$ is abelian, the kernel of $f$ contains $[\pi,\pi]$.
\end{proof}

For the next lemma, for $w \in \pi$ the notation $F^{w}$ means the $w$-conjugate of the subgroup $F$ of $G$.
Since $G$ is a normal subgroup of $G$, we have $F^w < G$.

\Figure{figure:gcover}{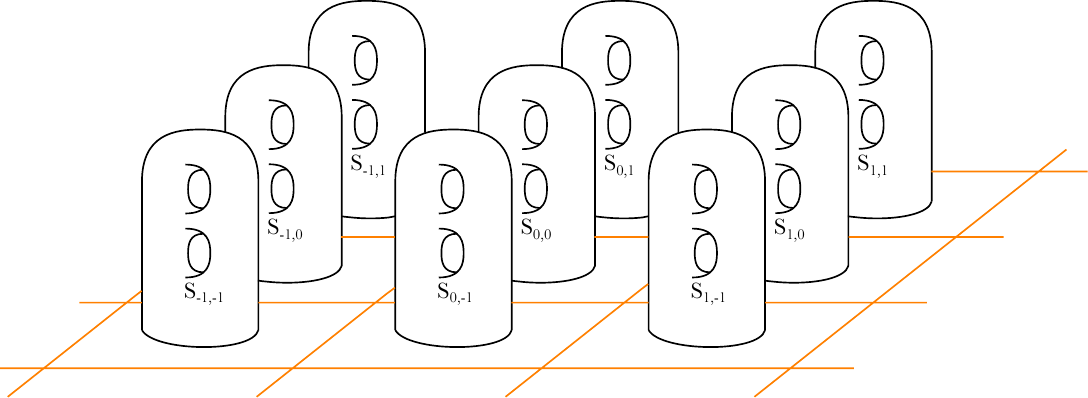}{The $\Z^2$-regular cover $\tSigma$ of $\Sigma_g$ corresponding to
the normal closure $G$ of $T = \{x_3,\ldots,x_{2g}\}$.  Putting the basepoint in $S_{0,0}$, the subgroup
$\pi_1(S_{0,0})$ of $G = \pi_1(\tSigma)$ is the subgroup $F$ generated by $T$.  The flat region is $X \subset \R^2$, and the lines
in the flat region are of the form
$(n+1/2) \times \R$ and $\R \times (n+1/2)$ for $n \in \Z$.}{82}

\begin{lemma}
\label{lemma:identifyg}
The subgroup $F < \pi$ is free on $T = \{x_3,\ldots,x_{2g}\}$, and
\[G = \BigFreeProd_{(k_1,k_2) \in \Z^2} F^{x_1^{k_1} x_2^{k_2}}.\]
\end{lemma}
\begin{proof}
Let $\ast \in \Sigma$ be the basepoint.  Let $(\tSigma,\tast)$ be the based cover of $(\Sigma_g,\ast)$ 
with $G = \pi_1(\tSigma,\tast)$.  We depict $\tSigma$ in Figure \ref{figure:gcover}.  As is shown there, it can be decomposed as
\[\tSigma = X \cup \left(\bigcup_{(k_1,k_2) \in \Z^2} S_{k_1,k_2}\right),\]
where $X$ and $S_{k_1,k_2}$ are as follows:
\begin{itemize}
\item $X = \R^2 \setminus \left(\cup_{(k_1,k_2) \in \Z^2} U_{k_1,k_2}\right)$, where $U_{k_1,k_2}$ is a small
open ball around $(k_1,k_2) \in \R^2$.
\item $S_{k_1,k_2}$ is a genus $(g-1)$ surface with $1$ boundary component glued to $X$ along
$\partial U_{k_1,k_2}$.
\end{itemize}
The deck group $\Z^2$ acts in the evident
way with $(n,m) \in \Z^2$ taking $S_{k_1,k_2}$ to $S_{k_1+n,k_2+m}$.

Choosing our identifications and basepoint correctly, we can ensure that $\tast \in S_{0,0}$ and that
$\pi_1(S_{0,0},\tast) < G$ is $F$.  It is then clear from our picture that $F$ is
free\footnote{This could also be deduced from the Freiheitsatz (\cite{MagnusFreiheitsatz}; see 
\cite{PutmanOneRelator} for an expository account).} on $T = \{x_3,\ldots,x_{2g}\}$ and that
\[G = \BigFreeProd_{(k_1,k_2) \in \Z^2} F^{x_1^{k_1} x_2^{k_2}},\]
as desired.
\end{proof}

Our next lemma is a general one and does not make use of the notation we introduced at
the beginning of this section.

\begin{lemma}
\label{lemma:basis}
Let $S$ be the set of formal symbols $\Set{$y_{n,m}$}{$n,m \in \Z$}$.  Then the free
group $F(S)$ is freely generated by
\[S' = \{y_{0,0}\} \cup \Set{$y_{n+1,m}^{-1} y_{n,m}$}{$n,m \in \Z$} \cup \Set{$y_{0,m+1}^{-1} y_{0,m}$}{$m \in \Z$}.\]
\end{lemma}
\begin{proof}
Consider the oriented tree $\fT$ from Figure \ref{figure:changebasis}.  This tree lies in $\R^2$ and its
vertices are $\Z^2$.  Define
\[S'' = \{y_{0,0}\} \cup \Set{$y_{n,m} y_{n',m'}^{-1}$}{$\fT$ has an oriented edge from $(n',m')$ to $(n,m)$}.\] 
The set $S'$ can be obtained from $S''$ by inverting some elements of $S''$, so it is enough to prove
that $F(S)$ is freely generated by $S''$.

\Figure{figure:changebasis}{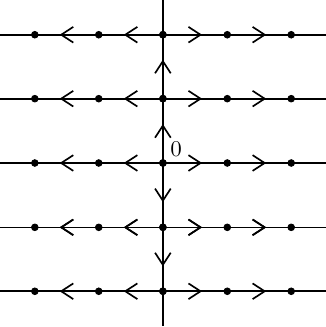}{The oriented tree $\fT \subset \R^2$ used in the proof of Lemma \ref{lemma:basis}.}{82}

Every $(n,m) \in \Z$ except for $(0,0)$ is the endpoint of precisely one edge in $\fT$, so we can define
a homomorphism $f\colon F(S) \rightarrow F(S)$ via the formula
\[f(y_{n,m}) = \begin{cases}
y_{n,m} y_{n',m'}^{-1} & \text{if there is an edge of $\fT$ from $(n',m')$ to $(n,m)$},\\
y_{0,0} & \text{if $(n,m) = (0,0)$}.
\end{cases}\]
We have $f(S) = S''$, so to prove that $S''$ freely generates $F(S)$ it is enough to prove that
$f$ is an automorphism of $F(S)$.

We will do this by writing down an inverse $g\colon F(S) \rightarrow F(S)$ for $f$.  For $(n,m) \in \Z^2$,
there is a unique geodesic edge path in $\fT$ from $(0,0)$ to $(n,m)$.  Let vertices traversed by this
path be
\[(0,0) = (a_1,b_1), (a_2,b_2), \ldots, (a_k,b_k) = (n,m).\]
We then define
\[w_{n,m} = y_{a_k,b_k} y_{a_{k-1},b_{k-1}} \cdots y_{a_1,b_1}.\]
Having done this, we define the homomorphism $g\colon F(S) \rightarrow F(S)$ via the formula
$g(y_{n,m}) = w_{n,m}$.  It is clear that $f \circ g = g \circ f = \text{id}$, so $g$ is the desired
inverse for $f$.
\end{proof}

Our final lemma uses the decomposition from Lemma \ref{lemma:identifyg}:

\begin{lemma}
\label{lemma:kernel}
Let
\[\phi\colon G = \BigFreeProd_{(k_1,k_2) \in \Z^2} F^{x_1^{k_1} x_2^{k_2}} \rightarrow F\]
be the homomorphism that takes the factor $F^{x_1^{k_1} x_2^{k_2}}$ to $F$ via conjugation by $x_1^{-k_1} x_2^{-k_2}$.
Then the kernel of $\phi$ is freely generated by
\begin{align*}
&\Set{$[x_1,x_j]^{x_1^{k_1} x_2^{k_2} w}$}{$3 \leq j \leq 2g$, $k_1,k_2 \in \Z$, and $w \in F$}\\
&\quad\quad\quad \cup \Set{$[x_2,x_j]^{x_2^{k_2} w}$}{$3 \leq j \leq 2g$, $k_2 \in \Z$, and $w \in F$}.
\end{align*}
\end{lemma}
\begin{proof}
Lemma \ref{lemma:identifyg} implies that $G$ has the free generating set
\[S = \Set{$x_j^{x_1^{k_1} x_2^{k_2}}$}{$3 \leq j \leq 2g$ and $k_1,k_2 \in \Z$}.\]
For each $x_j$ with $3 \leq j \leq 2g$, we can apply Lemma \ref{lemma:basis} to the free group on the subset
\[\Set{$x_j^{x_1^{k_1} x_2^{k_2}}$}{$k_1,k_2 \in \Z$} \subset S\]
by identifying $y_{n,m}$ with $x_j^{x_1^{n} x_2^{m}}$.  From this, we see that $G$ is freely generated
by
\begin{align*}
S' = \Set{$x_j$}{$3 \leq j \leq 2g$} &\cup \Set{$\left(x_j^{x_1^{k_1+1} x_2^{k_2}}\right)^{-1} x_j^{x_1^{k_1} x_2^{k_2}}$}{$3 \leq j \leq 2g$ and $k_1,k_2 \in \Z$}\\
&\cup \Set{$\left(x_j^{x_2^{k_2+1}}\right)^{-1} x_j^{x_2^{k_2}}$}{$3 \leq j \leq 2g$ and $k_2 \in \Z$}.
\end{align*}
Since
\begin{align*}
\left(x_j^{x_1^{k_1+1} x_2^{k_2}}\right)^{-1} x_j^{x_1^{k_1} x_2^{k_2}} &= x_2^{-k_2} x_1^{-k_1-1} x_j^{-1} x_1 x_j x_1^{k_1} x_2^{k_2} = [x_1,x_j]^{x_1^{k_1} x_2^{k_2}},\\
\left(x_j^{x_2^{k_2+1}}\right)^{-1} x_j^{x_2^{k_2}} &= x_2^{-k_2-1} x_j^{-1} x_2 x_j x_2^{k_2} = [x_2,x_j]^{x_2^{k_2}},
\end{align*}
we can rewrite $S'$ as
\begin{align*}
S' = \Set{$x_j$}{$3 \leq j \leq 2g$} &\cup \Set{$[x_1,x_j]^{x_1^{k_1} x_2^{k_2}}$}{$3 \leq j \leq 2g$, $k_1,k_2 \in \Z$}\\
&\cup \Set{$[x_2,x_j]^{x_2^{k_2}}$}{$3 \leq j \leq 2g$, $k_2 \in \Z$}.
\end{align*}
With respect to this free generating set for $G$, the homomorphism $\phi$ is the evident retract onto the
free group $F$ on $\{x_3,\ldots,x_{2g}\}$, so $\ker(\phi)$ is normally generated by the following subset of $S'$:
\begin{align*}
&\Set{$[x_1,x_j]^{x_1^{k_1} x_2^{k_2}}$}{$3 \leq j \leq 2g$, $k_1,k_2 \in \Z$}\\
&\quad\quad\quad \cup \Set{$[x_2,x_j]^{x_2^{k_2}}$}{$3 \leq j \leq 2g$, $k_2 \in \Z$}.
\end{align*}
Lemma \ref{lemma:easyargument} shows that the normal closure of the above set is freely generated by
\begin{align*}
&\Set{$[x_1,x_j]^{x_1^{k_1} x_2^{k_2} w}$}{$3 \leq j \leq 2g$, $k_1,k_2 \in \Z$, and $w \in F$}\\
&\quad\quad\quad \cup \Set{$[x_2,x_j]^{x_2^{k_2} w}$}{$3 \leq j \leq 2g$, $k_2 \in \Z$, and $w \in F$},
\end{align*}
as desired.
\end{proof}

We can now prove Theorem \ref{maintheorem:commsurfbasis}.

\begin{proof}[Proof of Theorem \ref{maintheorem:commsurfbasis}]
Recall that we must prove that the group $[\pi,\pi]$ is freely generated by
\[\Set{$[x_i,x_j]^{x_i^{k_i} x_{i+1}^{k_{i+1}} \cdots x_{2g}^{k_{2g}}}$}{$1 \leq i < j \leq 2g$, $(i,j) \neq (1,2)$, and $k_i,\ldots,k_{2g} \in \Z$}.\]
Lemma \ref{lemma:localcommutator} says that $[\pi,\pi] < G$.  Let $\phi\colon G \rightarrow F$
be the homomorphism from Lemma \ref{lemma:kernel}.  By construction, $\ker(\phi) \lhd [\pi,\pi]$
and $\phi([\pi,\pi]) = [F,F]$.  We thus have a short exact sequence
\[1 \longrightarrow \ker(\phi) \longrightarrow [\pi,\pi] \stackrel{\phi}{\longrightarrow} [F,F] \longrightarrow 1.\]
This exact sequence splits, so
\[[\pi,\pi] = \ker(\phi) \rtimes [F,F].\]
Lemma \ref{lemma:kernel} says that $\ker(\phi)$ is freely generated by
\begin{align*}
S = &\Set{$[x_1,x_j]^{x_1^{k_1} x_2^{k_2} w}$}{$3 \leq j \leq 2g$ and $k_1,k_2 \in \Z$ and $w \in F$}\\
    &\quad\quad\quad \cup \Set{$[x_2,x_j]^{x_2^{k_2} w}$}{$3 \leq j \leq 2g$ and $k_2 \in \Z$ and $w \in F$}.
\end{align*}
The group $[F,F]$ acts freely on $S$ by conjugation, and the set
\begin{align*}
S' = &\Set{$[x_1,x_j]^{x_1^{k_1} x_2^{k_2} \cdots x_{2g}^{k_{2g}}}$}{$3 \leq j \leq 2g$ and $k_1,\ldots,k_{2g} \in \Z$}\\
     &\quad\quad\quad \cup \Set{$[x_2,x_j]^{x_2^{k_2} x_3^{k_3} \cdots x_{2g}^{k_{2g}}}$}{$3 \leq j \leq 2g$ and $k_2,\ldots,k_{2g} \in \Z$}
\end{align*}
contains a single element from each $[F,F]$-orbit.  Lemma \ref{lemma:improvesemidirect} thus implies
that
\[[\pi,\pi] = F(S') \ast [F,F].\]
By Theorem \ref{maintheorem:commfnbasis}, the group $[F,F]$ is freely generated by
\[S'' = \Set{$[x_i,x_j]^{x_i^{k_i} x_{i+1}^{k_{i+1}} \cdots x_{2g}^{k_{2g}}}$}{$3 \leq i < j \leq 2g$ and $k_i,\ldots,k_{2g} \in \Z$}.\]
We conclude that $[\pi,\pi]$ is freely generated by $S' \cup S''$, as desired.
\end{proof}

\subsection{Module structure (Theorem \ref{maintheorem:commsurfmod})}
\label{section:modsurf}

We now prove Theorem \ref{maintheorem:commsurfmod}.

\begin{proof}[Proof of Theorem \ref{maintheorem:commsurfmod}]
Fix some $g \geq 1$.  Recall that this theorem asserts that there is a short
exact sequence
\[0 \longrightarrow \Z[\Z^{2g}] \longrightarrow [F_{2g},F_{2g}]^{\ab} \longrightarrow [\pi_1(\Sigma_g),\pi_1(\Sigma_g)]^{\ab} \longrightarrow 0\]
of $\Z[\Z^{2g}]$-modules.

Set
\[r = [x_1,x_2] \cdots [x_{2g-1},x_{2g}] \in [F_{2g},F_{2g}],\]
and let $\{r\} \in [F_{2g},F_{2g}]^{\ab}$ be the image of $r$.  Let $R$ be the $\Z[\Z^{2g}]$-span of $\{r\}$
in $[F_{2g},F_{2g}]^{\ab}$.  We have a short exact sequence
\[0 \longrightarrow R \longrightarrow [F_{2g},F_{2g}]^{\ab} \longrightarrow [\pi_1(\Sigma_g),\pi_1(\Sigma_g)]^{\ab} \longrightarrow 0\]
of $\Z[\Z^{2g}]$-modules, and to prove the theorem it is enough to prove that $R \cong \Z[\Z^{2g}]$.

Let $\Sigma_g^1$ be a compact oriented genus-$g$ surface with $1$ boundary component $\beta$.  Fix a basepoint
$\ast \in \beta$, and identify $F_{2g}$ with $\pi_1(\Sigma_g^1,\ast)$ in such a way that
$r$ is the loop around $\beta$.  Let $(\tSigma,\tast)$ be the cover of $(\Sigma_g^1,\ast)$ with
$\pi_1(\tSigma,\tast) = [F_{2g},F_{2g}]$.  We then have $\HH_1(\tSigma) = [F_{2g},F_{2g}]^{\ab}$.

By construction, $\{r\} \in \HH_1(\tSigma)$ is the homology class of the component of $\partial \tSigma$ containing
$\tast$.  The submodule $R$ of $\HH_1(\tSigma)$ is the image of $\HH_1(\partial \tSigma)$ in
$\HH_1(\tSigma)$.  The deck group $\Z^{2g}$ acts simply transitively on the set of components of $\partial \tSigma$, so
as a $\Z[\Z^{2g}]$-module we have
\[\HH_1(\partial \tSigma) \cong \Z[\Z^{2g}].\]
To prove the theorem, we must show that the map $\HH_1(\partial \tSigma) \rightarrow \HH_1(\tSigma)$ is injective.

This map fits into a long exact sequence of relative homology groups that contains the segment
\[\HH_2(\tSigma,\partial \tSigma) \longrightarrow \HH_1(\partial \tSigma) \longrightarrow \HH_1(\tSigma).\]
We must therefore prove that $\HH_2(\tSigma,\partial \tSigma) = 0$.
Collapse each component $\partial$ of $\partial \tSigma$ to a point $P_{\partial}$ to
form a surface $\tSigma'$.  Let 
\[P = \Set{$P_{\partial}$}{$\partial$ a component of $\tSigma$} \subset \tSigma'.\]
We then have 
\[\HH_2(\tSigma,\partial \tSigma) \cong \HH_2(\tSigma',P) \cong \HH_2(\tSigma') = 0,\]
where the final $=$ follows from the fact that $\tSigma'$ is a 
noncompact connected surface.  The theorem follows.
\end{proof}

\subsection{Homology (Theorem \ref{maintheorem:homologycalcsurf})}
\label{section:homologysurf}

We close the paper by proving Theorem \ref{maintheorem:homologycalcsurf}.

\begin{proof}[Proof of Theorem \ref{maintheorem:homologycalcsurf}]
Fix some $g \geq 1$.  We must prove that
\[\HH_k(\Z^{2g};[\pi_1(\Sigma_g),\pi_1(\Sigma_g)]^{\ab}) \cong
\begin{cases}
(\wedge^2 \Z^{2g}) / \Z & \text{if $k = 0$},\\
\wedge^{k+2} \Z^{2g} & \text{if $k \geq 1$}.
\end{cases}\]
By Theorem \ref{maintheorem:commsurfmod}, we have a short exact sequence
\begin{equation}
\label{eqn:ourlong}
0 \longrightarrow \Z[\Z^{2g}] \longrightarrow [F_{2g},F_{2g}]^{\ab} \longrightarrow [\pi_1(\Sigma_g),\pi_1(\Sigma_g)]^{\ab} \longrightarrow 0
\end{equation}
of $\Z[\Z^{2g}]$-modules.  Since $\Z[\Z^{2g}]$ is a free $\Z[\Z^{2g}]$-module, we have
\[\HH_k(\Z^{2g};\Z[\Z^{2g}]) = \begin{cases}
\Z & \text{if $k=0$},\\
0  & \text{if $k \geq 1$}.\end{cases}\]
Also, Theorem \ref{maintheorem:homologycalcsurf} says that
\[\HH_k(\Z^{2g};[F_{2g},F_{2g}]^{\ab}) \cong
\begin{cases}
\wedge^2 \Z^{2g} & \text{if $k = 0$},\\
\wedge^{k+2} \Z^{2g} & \text{if $k \geq 1$}.
\end{cases}\]
The long exact sequence in $\Z^{2g}$-homology associated to \eqref{eqn:ourlong} thus immediately implies
the result we want for $k \geq 2$.  For $k=0,1$, this long exact sequence contains the segment
\begin{align*}
&0 \rightarrow \wedge^3 \Z^{2g} \rightarrow \HH_1(\Z^{2g};[\pi_1(\Sigma_g),\pi_1(\Sigma_g)]^{\ab}) \\
&\quad\quad\quad \rightarrow \Z \rightarrow \wedge^2 \Z^{2g} \rightarrow \HH_0(\Z^{2g};[\pi_1(\Sigma_g),\pi_1(\Sigma_g)]^{\ab}) \rightarrow 0.
\end{align*}
To prove the theorem, we must therefore prove that the map $\Z \rightarrow \wedge^2 \Z^{2g}$ in this
exact sequence is not the zero map.  For $1 \leq i \leq 2g$, let $\ox_i \in \Z^{2g}$ be the image
of $x_i \in F_{2g}$ in $(F_{2g})^{\ab} = \Z^{2g}$.  Tracing through all the maps involved, we see
that the map $\Z \rightarrow \wedge^2 \Z^{2g}$ takes the generator of $\Z$ to
\[\ox_1 \wedge \ox_2 + \cdots + \ox_{2g-1} \wedge \ox_{2g} \in \wedge^2 \Z^{2g}.\]
This is indeed nonzero, and the theorem follows.
\end{proof}

\begin{footnotesize}
\noindent
\begin{tabular*}{\linewidth}[t]{@{}p{\widthof{School of Mathematics}+0.3in}@{}p{\widthof{Department of Mathematics}+0.3in}@{}p{\linewidth - \widthof{Department of Mathematics} - \widthof{School of Mathematics}- 0.6in}@{}}
&&{\raggedright
Andrew Putman\par
Department of Mathematics\par
University of Notre Dame \par
255 Hurley Hall\par
Notre Dame, IN 46556\par
{\tt andyp@nd.edu}}
\end{tabular*}
\end{footnotesize}

\end{document}